\documentclass[12pt]{amsart}
\usepackage{graphicx} 
\usepackage{amsfonts}
\usepackage{amsmath}
\usepackage{mathtools}
\usepackage{amssymb}
\usepackage{amsthm}
\usepackage{hyperref}
\usepackage{xcolor}
\usepackage{fullpage}
\usepackage[all,cmtip]{xy}

\DeclareMathOperator{\Gal}{Gal}
\DeclareMathOperator{\GL}{GL}
\DeclareMathOperator{\GSp}{GSp}
\DeclareMathOperator{\GU}{GU}
\DeclareMathOperator{\GO}{GO}

\DeclareMathOperator{\PGL}{PGL}
\DeclareMathOperator{\End}{End}
\DeclareMathOperator{\Aut}{Aut}
\DeclareMathOperator{\Ind}{Ind}

\newcommand{\ZZ}{\mathbb{Z}}
\newcommand{\QQ}{\mathbb{Q}}
\newcommand{\PP}{\mathbb{P}}
\newcommand{\CC}{\mathbb{C}}
\newcommand{\C}{\mathcal{C}}

\newcommand{\A}{\mathbb{A}}
\newcommand{\Tl}{\mathrm{T}_\ell}
\newcommand{\Vl}{\mathrm{V}_\ell}

\newcommand{\rhoA}{\rho_{A,\lambda}}
\newcommand{\rhoAt}{\widetilde{\sigma_{A,\lambda}}}

\newtheorem{theorem}{Theorem}[section]
\newtheorem{proposition}[theorem]{Proposition}
\newtheorem{lemma}[theorem]{Lemma}
\newtheorem{corollary}[theorem]{Corollary}
\newtheorem{definition}[theorem]{Definition}

\newtheorem{remark}[theorem]{Remark}

\newcommand{\lmfdbgenustwo}[3]{\href{http://www.lmfdb.org/Genus2Curve/Q/#1/#2/#1/#3}{#1-#2-#3}}

\begin{document}
\title{K-varieties and Galois representations}

\author{Enric Florit}

\address{Departament de Matemàtiques i Informàtica, Universitat de
  Barcelona, Gran via de les Corts Catalanes, 585, 08007 Barcelona,
  Spain}
\email{enricflorit@ub.edu}

\author{Ariel Pacetti}

\address{Center for Research and Development in Mathematics and
  Applications (CIDMA), Department of Mathematics, University of
  Aveiro, 3810-193 Aveiro, Portugal} \email{apacetti@ua.pt}
\thanks{Funded by the Portuguese Foundation for
  Science and Technology (FCT) Individual Call to Scientific Employment Stimulus
  (https://doi.org/10.54499/2020.02775.CEECIND/CP1589/CT0032). This
  work was supported by CIDMA and is funded by the FCT, under grant
  UIDB/04106/2020 (https://doi.org/10.54499/UIDB/04106/2020). Florit was supported by the Spanish Ministry of Universities (FPU20/05059) and by grants PID2022-137605NB-I00 and 2021 SGR 01468.}


\begin{abstract}
  In a remarkable article Ribet showed how to attach rational
  $2$-dimensional representations to elliptic $\QQ$-curves. An abelian
  variety $A$ is a (weak) $K$-variety if it is isogenous to all of its
  $\Gal_K$-conjugates. In this article we study the problem of
  attaching an absolutely irreducible $\ell$-adic representation of
  $\Gal_K$ to an abelian $K$-variety, which sometimes has smaller
  dimension than expected. When possible, we also construct a
  Galois-equivariant pairing, which restricts the image of this
  representation. As an application of our construction, we prove
  modularity of abelian surfaces over $\QQ$ with potential
  quaternionic multiplication.
\end{abstract}

\maketitle

\section*{Introduction}
Let $L$ be a number field, and let $A/L$ be an abelian variety of
dimension $g$. Given a prime number $\ell$, the Tate module $T_\ell A$
is obtained as the inverse limit (over all positive integers $n$) of
the subgroup of $\ell^n$-torsion points of $A$. Since addition is a
rational map, the Tate module acquires a structure of
$\ZZ_\ell[\Gal_L]$-module. Let
$\Vl A:= \Tl A\otimes_{\ZZ_\ell}\QQ_\ell$ be the usual extension
of scalars, a $\QQ_\ell$-vector space of dimension $2g$. The action of
$\Gal_L$ on $\Vl A$ gives the usual $2g$-dimensional Galois
representation $\rho_{A,\ell}:\Gal_L \to \GL_{2g}(\QQ_\ell)$. Let
$\End^0(A)$ denote the $\QQ$-algebra of endomorphisms of $A$ defined
over $L$. A well known result states that if $\End^0(A)$ is larger
than $\QQ$, then there exists a subrepresentation $\rho_{A,\lambda}$
of $\rho_{A,\ell}$ obtained by considering $\Vl A$ as an
$\End^0(A)\otimes_\QQ \QQ_\ell$-module (we recall this construction  in
\S~\ref{section:endomorphisms}).

Let $L/K$ be a Galois extension of number fields. The variety $A$ is
called a (weak) $K$-variety if it is $L$-isogenous to all of its
$\Gal(L/K)$-conjugates, namely for all $\sigma \in \Gal(L/K)$ there
exists an isogeny $\mu_\sigma: {}^\sigma A \to A$. The prototypical
example of a $K$-variety is that of an elliptic curve $A$ which is
isogenous to all of its Galois conjugates (the so called $\QQ$-curves)
as studied by Ribet in \cite{MR2058653}. The main goal of the present
article is to study the following question.

\vspace{3pt}
{\bf \noindent Question:} What Galois representations can
be naturally attached to a (weak) $K$-variety?
\vspace{3pt}

The reason to study $K$-varieties is that although the variety $A$ is
defined over $L$, its Galois representation (up to a twist) should
extend to an $n$-dimensional Galois representation of $\Gal_K$ for
some $n \le 2\dim(A)$ (this is indeed the case, as proven in
Theorem~\ref{thm:irreducibility-dimension}). In the particular case of
$\dim(A)=1$ (i.e. $A$ being an elliptic curve) without complex
multiplication, this result was proven by Ribet in \cite{MR2058653}.
Existence of endomorphisms makes the situation more subtle.  It is
important to remark that as part of the Langlands program, we also
expect the resulting Galois representation of $\Gal_K$ to match an
automorphic representation of $\GL_n(\A_K)$. In the case of
$\QQ$-curves, this follows from Serre's conjectures (as described in
\cite{MR2058653}).

An abelian variety $A/L$ is called a \emph{strong} $K$-variety if
furthermore for all $\sigma \in \Gal(L/K)$ the isogeny $\mu_\sigma:{}^\sigma A\to A$ satisfies a natural commutation
relation with all endomorphisms defined over $L$, namely 
\[
  \mu_\sigma {}^\sigma \varphi = \varphi \mu_\sigma \; \forall \sigma \in \Gal(L/K), \; \forall \varphi \in \End(A_L).
  \]
  We want to emphasize that sometimes in the literature, the notion of
  a $K$-variety involves the stronger condition that the
  previous relation holds for all
  $\varphi \in \End(A_{\overline{\QQ}})$. Strong $K$-varieties (with the more restrictive definition) were
  studied in \cite{guitart12}, in analogy to the $\QQ$-curves
  situation. It is not hard to prove that if $A/L$ is a strong
  $K$-variety of dimension $g$ with endomorphism ring $\ZZ$, then a
  twist of $\rho_{A,\ell}$ can be extended to a $2g$-dimensional
  Galois representation of $\Gal_K$ (see
  Theorem~\ref{theorem:extension}). In some instances (see
  Remark~\ref{remark:strong-bilinear}), the extended representation
  also preserves a non-degenerate bilinear form.

  In general, the $\ell$-adic representation $\rho_{A,\ell}$ of the
  Galois group $\Gal_L$ has an irreducible subrepresentation
  $\rho_{\lambda}$ (coming from the endomorphism ring). Then we can
  prove (Theorem~\ref{thm:strong-field}) the existence of a maximal
  subfield $L' \subset L$ such that $A$ is a strong $L'$-variety. This
  allows to extend the representation $\rho_{A,\lambda}$ to a
  representation $\widetilde{\rho_{\lambda}}$ of $\Gal_{L'}$ of the
  same dimension (well defined up to twists by characters of
  $\Gal_{L'}$). To obtain a representation of $\Gal_K$ one can just
  consider its induction from $\Gal_{L'}$ to $\Gal_K$ (as done in
  \cite{MR4381226} for strong abelian varieties of $\GL_2$-type over
  $\overline{\QQ}$).  As an example, let $E/\QQ$ be an elliptic curve
  with complex multiplication by an imaginary quadratic field $K$. The
  curve $E/K$ is clearly a $K$-elliptic curve, though it is not a
  strong $K$-elliptic curve (which is consistent with the fact that it
  does not have a $1$-dimensional representation of $\Gal_\QQ$
  attached to it). The representation attached to $E/K$ decomposes as
  the sum of two $1$-dimensional representations (coming from a Hecke
  character $\chi$ and its complex conjugate), and the induction of
  $\chi$ from $\Gal_K$ to $\Gal_\QQ$ matches $\rho_{E,\ell}$.

  Our construction is well suited to study how the Galois
  representation of $A$ behaves while extending scalars from $L$ to a
  field $M$ where it gets extra endomorphisms.  In
  Theorem~\ref{thm:dimension} we compare the dimension of the
  different constructions, which do not always match. We present two
  interesting applications of our construction: one of them is related
  to abelian fourfolds whose $4$-dimensional Galois representation is
  contained in $\GU_4(\QQ_\ell) \cap \GSp_4(F_\lambda)$, for $F$ an
  imaginary quadratic field, and $\lambda$ a prime of $F$ dividing
  $\ell$. Our second application proves that any abelian surface over
  $\QQ$ with potential quaternionic multiplication (QM for short)
  comes from a Siegel modular form of weight $2$.

  The article is organized as follows:
  Section~\ref{section:endomorphisms} recalls some well known facts on
  irreducible constituents $\rho_{\lambda}$ of Galois representations
  attached to abelian varieties in terms of its endomorphism
  ring. Section~\ref{section:pairings-and-Albert-types} recalls
  results on the existence of non-degenerate, bilinear pairings
  (symplectic, symmetric or hermitian) invariant under the action of
  $\rho_{\lambda}$. Section~\ref{section:K-varieties} contains the
  main results of the article. We start studying strong $K$-varieties
  $A/L$. Our main result (Theorem~\ref{theorem:extension}) proves that
  the representation $\rho_{\lambda}$ (after possible a twist) can be
  extended to a representation $\widetilde{\rho_{\lambda}}$ of
  $\Gal_K$ of the same dimension. In
  Section~\ref{subsection:general-k-varieties} we study general
  $K$-varieties. We prove that the group $\Gal(L/K)$ has a natural
  action on the center of the endomorphism algebra. The kernel of the
  action fixes a subextension $L'$, and it happens that the variety
  $A$ is a strong $L'$-variety (see
  Theorem~\ref{thm:strong-field}). Then the previous result applies
  and provides an extension of our representation to $\Gal_{L'}$. In
  Section~\ref{subsection:dimension} we study how the dimension our
  constructed representation varies when we enlarge the base field
  $L$. In Theorem~\ref{thm:dimension2} we prove that in some cases the
  dimension halves (depending on how many new extra endomorphisms
  the variety gains).

Section~\ref{section:pairings} gives the necessary (and sufficient)
condition for the existence of an invariant pairing for the
constructed representation extending that of $\rho_\lambda$. The last
section contains some applications (including the stated modularity result of abelian
surfaces with potential QM). 

\vspace{3pt}
\noindent{\bf Notation:} for the reader's convenience, we include some
notation used during the article.
\begin{itemize}
\item $L$ denotes a number field and $\overline{L}$ an algebraic closure of it.
  
\item $A$ denotes a simple abelian variety defined over $L$.
  
\item $R = \End(A_L)$ and $D = \End^0(A) = R \otimes_\ZZ \QQ$. $E$ denotes the center
  of $D$ and $E^+$ its maximal totally real subfield.
  
\item $d$ denotes the Schur index of $\End^0(A)$.
  
\item If $\varphi \in D$ and $\sigma \in \Gal_K$, we denote by
  ${}^\sigma \varphi$ the morphism obtained by applying $\sigma$ to
  the coefficients of the endomorphism.
\end{itemize}

\vspace{2pt}
\noindent{\bf Acknowledgments:} We would like to thank Alex Bartel,
Francesc Fité and Xevi Guitart for many helpful conversations as well
as many suggestions that improved the quality of the article.

\section{Endomorphisms and splitting}
\label{section:endomorphisms}
Let $L$ be a number field, and let $A/L$ be a simple abelian variety of dimension $g$. Let
$\End^0(A) := \End(A_L)\otimes_\ZZ \QQ$ be its endomorphism algebra, a
division algebra with center a number field $E$. Then
$\dim_E \End^0(A) = d^2$, where the number $d$ is called the Schur
index of $\End^0(A)$ (see \cite[\S13]{MR0674652}).
\begin{theorem}\label{theorem:albert}
  The endomorphism algebra falls in one of the following types in the Albert
  classification:
\begin{itemize}
\item Type {\rm I}. $E=\End^0(A)$ (so $d=1$); $E$ is a totally real number
  field and $[E:\QQ]\mid g$.
\item Type {\rm II}. $E$ is a totally real number field, $\End^0(A)$
  is a totally indefinite quaternion algebra over $K$ (so $d=2$) and
  $2[E:\QQ]\mid g$.
\item Type {\rm III}. $E$ is a totally real number field, $\End^0(A)$
  is a totally definite quaternion algebra over $K$ (so $d=2$) and
  $2[E:\QQ]\mid g$.
\item Type {\rm IV}. $E$ is a CM field, $\End^0(A)$ is a division algebra and $\frac{[E:\QQ]}{2}d^2\mid g$.
\end{itemize}
\end{theorem}
\begin{proof}
  See for example Theorem 2 of \cite[\S~21]{mumford74}.
\end{proof}
\begin{remark}
  Even when $A_{\bar L}$ is simple, it is perfectly possible that $A_L$
  and $A_{\bar L}$ have different Albert types. For example a
  rational curve with complex multiplication by an imaginary quadratic
  field $K$ has Albert type {\rm I} over $\QQ$, but type {\rm IV} over $K$. The
  same phenomenon occurs in higher dimensions, for example with the Jacobian of
  the genus 2 curve with LMFDB label \lmfdbgenustwo{20736}{a}{1} (with complex
  multiplication over $\QQ(\zeta_9)^+$).
\end{remark}

Fix a prime number $\ell$.  Let
$\Vl(A):=\Tl(A)\otimes_{\ZZ_\ell} \QQ_\ell$ be the rational Tate
module. Fix a polarization $\phi:A \to A^\vee$. The Weil pairing and
the choice of polarization gives a non-degenerate alternating pairing
\[
  \Phi: \Vl A\times \Vl A\to \QQ_\ell.
\]
This pairing is compatible with the action of Galois: for
$\sigma\in\Gal_L$, $u,v\in\Vl A$, we have
\begin{equation}
  \label{eq:Weil}
  \Phi({}^\sigma u,{}^\sigma v) = {}^\sigma
  \Phi(u,v)=\chi_\ell(\sigma) \Phi(u,v)
\end{equation}
with $\chi_\ell:\Gal_L\to\QQ_\ell^\times$ the $\ell$-adic cyclotomic
character. Another way to state property (\ref{eq:Weil}) is to say
that $\Phi$ is $\Gal_L$-equivariant with similitude character
$\chi_\ell$. The pairing (and its $\Gal_L$-invariance) allow to
construct a continuous Galois representation
\[
  \rho_{A,\ell}:\Gal_L\to\GSp(\Vl(A),\Phi),
\]
with similitude character $\chi_\ell$, the $\ell$-th cyclotomic
character. This representation is not irreducible in general.
To ease notation, let $R := \End(A_L)$ and let
$D:=\End^0(A) = R\otimes_\ZZ \QQ$. When $D$ is non-commutative define the
ramification set
\begin{equation}
  \label{eq:ramification}
  \mathrm{Ram}(D)=\{ \lambda\text{ prime of }E\ \colon D\otimes_E
  E_\lambda\not\simeq M_d(E_\lambda) \}.
\end{equation}
To avoid multiple statements, when $D=E$ set
$\mathrm{Ram}(D)=\emptyset$. In both cases, $\mathrm{Ram}(D)$ is a
finite set with even cardinality \cite[\S18.5]{MR0674652}. The algebra
$\End^0(A)$ acts on $\Vl A$, so we can consider it as an
$\End^0(A)\otimes_\QQ \QQ_\ell$-module, and in particular as an
$E\otimes_\QQ \QQ_\ell$-module.  There is a (categorical product)
decomposition
\begin{equation}
  \label{eq:E-decomp}
E\otimes_\QQ\QQ_\ell \simeq \bigtimes_{\lambda\mid\ell}E_\lambda,  
\end{equation}
where $\lambda$ runs over the primes of $E$ over $\ell$. Let
$e_\lambda$ be the idempotent corresponding to $E_\lambda$ in (\ref{eq:E-decomp}) and set
$\mathrm{V}_\lambda := e_\lambda\cdot \Vl A$. Then
\[
  \Vl A\simeq \bigoplus_{\lambda\mid\ell} \mathrm{V}_\lambda.
\]
This is an isomorphism of $E\otimes_\QQ \QQ_\ell[\Gal_L]$-modules, since
$\End^0(A)$ is the algebra of endomorphisms defined over $L$, so
they commute with the action of the Galois group $\Gal_L$.

\begin{lemma}
  We have
  $\dim_{E_\lambda}\mathrm{V}_\lambda=\frac{2g}{[E:\QQ]}$, which is
  independent of $\ell$.
\end{lemma}
\begin{proof}
	This is Theorem 2.1.1 in \cite{MR0457455}, the argument works the same since $\mathrm{V}_\lambda$ is a free $E_\lambda$-module.
\end{proof}

\begin{lemma}
  We have
  $\End_{E_\lambda[\Gal_L]}(\mathrm{V}_\lambda)=\End^0(A)\otimes_E
  E_\lambda$. Furthermore, if $\End^0(A)=E$ is commutative, each
  $\mathrm{V}_\lambda$ is an absolutely irreducible representation of
  $\Gal_L$.
\end{lemma}
\begin{proof}
  By \cite[ Satz 4]{MR0718935}, $\Vl A$ is a semisimple
  $\QQ_\ell[\Gal_L]$-module and
  $\End_{\Gal_L}(\Vl A)\simeq \End^0(A)\otimes_\QQ \QQ_\ell$. Therefore
  each $\mathrm{V}_\lambda$ is semisimple and satisfies
  $\End_{E_\lambda[\Gal_L]}(\mathrm{V}_\lambda)=\End^0(A)\otimes_E
  E_\lambda$.  Hence when $\End^0(A)=E$, $\mathrm{V}_\lambda$ is
  simple over $E_\lambda$, and so it is absolutely irreducible.
\end{proof}

When $\End^0(A)$ is non-commutative, the module $\mathrm{V}_\lambda$
cannot be simple. The reason is that
$\End^0(A)\otimes_E E_\lambda \simeq M_d(E_\lambda)$ for
$\lambda\not\in\mathrm{Ram}(D)$, and this ring has $d$ orthogonal
idempotents that further break down $\mathrm{V}_\lambda$. Since we
want the simple factors in a decomposition of $\mathrm V_\lambda$ to
be nondegenerate with respect to a certain pairing, we introduce them
in a separate section.

\section{Pairings and irreducible constituents}
\label{section:pairings-and-Albert-types}

\begin{definition}
  Let $K$ be a field of characteristic different from 2 and let $V$ be
  a finite dimensional $K$-vector space. A nondegenerate biadditive
  form
    \[
    \Psi:V\times V\to K
    \]
    is called
    \begin{itemize}
    \item \emph{symplectic}, if it is $K$-bilinear and $\Psi(v,v)=0$
      for all $v\in V$.
    \item \emph{symmetric}, if it is $K$-bilinear and
      $\Psi(v,w)=\Psi(w,v)$ for all $v,w\in V$.
    \item \emph{hermitian}, if $K$ has an automorphism
      $\bar\cdot:K\to K$ which is an involution, and satisfies that
      $\Psi$ is $K$-linear in the first entry and
    \[
    \Psi(v,w)=\overline{\Psi(w,v)}
    \]
    for all $v,w\in V$.
    \end{itemize}
  \end{definition}

\begin{definition}
  Suppose that $A$ either has Albert type {\rm I}, {\rm II} or {\rm III}. Let $\ell$ be a
  rational prime and let $\lambda\mid\ell$ in $\mathcal O_E$. The
  prime $\lambda$ has \emph{Property (P)} if $\lambda\not\in\mathrm{Ram}(D)$.
\end{definition}

Denote by $x\mapsto x'$ the Rosati involution of $D=\End^0(A)$,
which is obtained from the polarization $\phi:A\to A^\vee$.

\begin{lemma}
  Let $A$ have Albert type {\rm IV}. Let $E^+$ be the maximal totally real
  subfield of $E$. Then there exist a finite Galois extension $L/E^+$
  containing $E$, an element $\gamma\in D$ with $\gamma'=\gamma$ and
  an $L$-algebra isomorphism
  \[
    s: D\otimes_E L\overset{\sim}{\to} M_d(L)
  \]
  such that the positive involution
  $x\mapsto x^*:=\gamma x'\gamma^{-1}$ on $s(D)$ is the restriction of
  the involution $X\mapsto X^*:=\overline{X}^{\top}$ on $M_d(L)$.
\end{lemma}
\begin{proof}
	This is Lemma~2.1 in \cite{bk-IV}.
\end{proof}

\begin{definition}
  Suppose $A$ has Albert type {\rm IV} with $d>1$. Let $\ell$ be a rational
  prime, let $\lambda\mid\ell$ in $\mathcal O_E$, and let
  $\lambda_0:=\lambda\cap \mathcal O_{E^+}$. The prime $\lambda$ has
  \emph{Property (P)} if it satisfies the following properties:
	\begin{itemize}
        \item $\lambda\not\in\mathrm{Ram}(D)$.
        \item $\lambda$ is inert over $\lambda_0$, and $\lambda$
          splits in $\mathcal O_L$.
	\end{itemize}
\end{definition}

\begin{remark}
As explained in \cite[pg. 1250]{bk-IV}, when $A$ has Albert type {\rm IV} there is a positive density of primes $\lambda$ of $E$ with property (P). 
\end{remark}

\begin{theorem}
\label{theorem:pairing-existence}
Let $A/L$ be a simple abelian variety of dimension $g$, and let
$n = \frac{2g}{d[E:\QQ]}$. Let $\ell$ be a prime number with Property
(P), and let $\lambda\mid \ell$ in $E$. Then there exist an absolutely
irreducible $E_\lambda[\Gal_L]$-module
$\mathrm{W}_\lambda$ of rank $n$, and a pairing
$\Psi_\lambda:\mathrm{W}_\lambda\times\mathrm{W}_\lambda\to
E_\lambda$, satisfying the following properties:
\begin{itemize}
\item The pairing $\Psi_\lambda$ is non-degenerate and $E_\lambda$-linear in the first coordinate.
\item The pairing is symplectic if $A$ is of type {\rm I} or {\rm II}.
  
\item The pairing is symmetric if $A$ is of type {\rm III}.

\item The pairing is hermitian if $A$ is of type {\rm IV}.

\item The form $\Psi$ is invariant (up to a similitude character) under
  the Galois action: for $u,v\in\mathrm{W}_\lambda$ and
  $\sigma\in\Gal_L$,
  $\Psi_\lambda({}^\sigma u,{}^\sigma
  v)={}^\sigma\Psi_\lambda(u,v)=\chi_\ell(\sigma)\Psi_\lambda(u,v)$.
  
\item There is an isomorphism
  \[
    V_\lambda \simeq \mathrm{W}_\lambda^{\oplus d}
  \]
  of $E_\lambda[\Gal_L]$-modules. Furthermore, if every
  prime ideal $\lambda\mid\ell$ of $E$ has Property (P), then
  there is a decomposition
\[
	\Vl A \simeq \bigoplus_{\lambda'\mid\ell}\mathrm{W}_{\lambda'}^{\oplus d}
\]
of $E\otimes_\QQ \QQ_\ell[\Gal_L]$-modules.
\end{itemize}
\end{theorem}

\begin{proof}
  This is a summary of \cite[Theorem~5.4]{bgk-I-II} and
  \cite[Theorem~3.23]{bgk-III}. If $A$ has Albert type I, then this
  follows from \cite[\S~2.3]{MR1120395} if $d=1$, and from
  \cite[Theorem~1.1]{bk-IV} when $d>1$.
\end{proof}

\section{$K$-varieties}
\label{section:K-varieties}

\begin{definition}
  Let $L/K$ be an extension of number fields and let $A$ be an abelian
  variety defined over $L$. The variety $A$ is called a
  \emph{$K$-variety} if for every $\sigma\in \Gal_K$ there exists an
  isogeny $\mu_\sigma:{}^\sigma A\to A$. The variety $A$ is called a
  \emph{strong $K$-variety} if, in addition, for every $\varphi\in \End(A_L)$
  the following equality holds
  \begin{equation}
    \label{eq:strong-k-var}
  \mu_\sigma {}^\sigma \varphi = \varphi \mu_\sigma.
  \end{equation}
\end{definition}
\begin{remark}
  If $L'/L$ is a field extension, and $A/L$ is a $K$-variety, then the
  variety $A/L'$ is also a $K$-variety. This is not true for strong
  $K$-varieties. For example if $E$ is an elliptic curve over $\QQ$
  with CM by an imaginary quadratic field $K$, then $E/\QQ$ is clearly
  a strong $\QQ$-curve, but $E/K$ is not.
\end{remark}

\subsection{Strong $K$-varieties}
Let $A$ be a strong $K$-variety defined over a number field
$L$. Following the notation of
Theorem~\ref{theorem:pairing-existence}, let $\rho_{\lambda}$ be the
Galois representation attached to $A/L$.

\begin{theorem}\label{theorem:extension}
  Suppose that the extension $L/K$ is Galois. Then there exists a
  finite order character $\psi:\Gal_L \to \overline{\QQ_\ell}^\times$
  such that the twisted representation $\rho_{\lambda} \otimes \psi$
  extends to a representation of $\Gal_K$.
\end{theorem}
\begin{proof}
The proof is quite standard (see for example Theorem 11.2 of \cite{MR2270898}).
  For $\tau \in \Gal(L/K)$ let $^\tau\rho_{\lambda}$ denote the
  representation defined by
  \[
^\tau\rho_\lambda(\sigma) = \rho_\lambda(\tau \sigma \tau^{-1}).
  \]
  Use the same notation for the Tate representation
  $\rho_{A,\ell}$. The fact that $A$ is a $K$-variety implies that
  $^\tau \rho_{A,\ell}$ is isomorphic to $\rho_{A,\ell}$. The fact
  that $A$ is a strong $K$-variety (i.e. the fact that isogenies
  commute with endomorphisms as in (\ref{eq:strong-k-var})), imply
  that the same holds for the representation $\rho_\lambda$, namely
  $^\tau\rho_\lambda \simeq \rho_\lambda$.  Then there exists a matrix
  $A_\tau \in \GL_n(\overline{\QQ_\ell})$ such that for all
  $\sigma \in \Gal_L$
  \[
^\tau\rho_\lambda(\sigma) = A_\tau \rho_\lambda(\sigma)A_\tau^{-1}.
\]
Since the representation $\rho_\lambda$ is irreducible (by
Theorem~\ref{theorem:pairing-existence}), Schur's lemma implies that
the matrix $A_\tau$ is unique up to scalars, i.e. it determines a
unique element in $\PGL_n(\overline{\QQ_\ell})$.

Consider the projective representation
$\PP\rho_{\lambda}:\Gal_L \to \PGL_n(\overline{\QQ_\ell})$ (obtained
as the composition of $\rho_\lambda$ with the natural quotient map
$\pi:\GL_n(\overline{\QQ_\ell}) \to \PGL_n(\overline{\QQ_\ell})$).
Extend the representation $\PP\rho_\ell$ to a map
$\widetilde{\PP\rho_\ell}:\Gal_K \to \PGL_n(\overline{\QQ_\ell})$ by
defining $\widetilde{\PP\rho_\lambda(\tau)} := A_\tau$ for
$\tau \in \Gal_K$. An elementary computation proves that
$\widetilde{\PP\rho_\lambda}$ is a group morphism (hence a projective
representation).  Note that if $\tau \in \Gal_L$, then
\[
^\tau\rho_\lambda(\sigma) = \rho_{\lambda}(\tau \sigma \tau^{-1}) = \rho_{\lambda}(\tau)\rho_{\lambda}(\sigma)\rho_{\lambda}(\tau)^{-1},
\]
hence $A_\tau = \rho_{\lambda}(\tau)$ (up to a scalar matrix), so the map
$\widetilde{\PP\rho_\ell}$ really extends the map $\PP\rho_\ell$.

The obstruction to lift the projective representation to a continuous
homomorphism $\rho:\Gal_K \to \GL_2(\overline{\QQ_\ell})$ lies in
$H^2(\Gal_K,\overline{\QQ_\ell^\times})$, which by a result of Tate
(see \cite[\S 6.5]{MR450201}) is trivial.
\end{proof}

\begin{remark}
  It follows from the proof of the previous theorem that the extension
  is unique up to a twist by a character of $\Gal_K$. Note that when
  $L/K$ is cyclic, a similar argument proves that $\rho_\lambda$
  admits an extension without needing to twist.
\label{remark:strong-unicity}
\end{remark}

\begin{remark}
  When $L/K$ is not cyclic, the twist might be really needed. The
  following elementary example was provided to us by Alex Bartel: let
  $Q$ be the group of quaternions, and let $H$ be the subgroup
  $H=\{\pm 1\}$. Let $\rho:H \to \CC^\times$ the representation
  sending $-1$ to $-1$. It is clear that this representation cannot be
  extended to a $1$-dimensional representation of $Q$ (because the
  commutator subgroup of $Q$ is $H$), but it is true that for any
  $g \in Q$, the representation $^q\rho$ is isomorphic to $\rho$
  (because $H$ is the center of $Q$).
\end{remark}

We will denote by $\widetilde{W_\lambda}$ the vector space underlying
the extended representation $\widetilde{\rho_\lambda}$, and by
$\widetilde{E_\lambda}$ its coefficient field (so
$\dim_{E_\lambda}W_\lambda =
\dim_{\widetilde{E_\lambda}}\widetilde{W_\lambda}$).

\subsection{General $K$-varieties}
\label{subsection:general-k-varieties}
Let $L/K$ be an extension of number
fields, which we assume is Galois (otherwise enlarge $L$). Let $A/L$
be an isotypic $K$-variety, so $A\sim B^n$ for some simple abelian
variety $B$.  Let $D = \End(A_L)\otimes_\ZZ \QQ$ be the endomorphism
algebra of $A$ and let $E$ be its center, so $E$ is a field and $D$ is
a simple central $E$-algebra.  For $\sigma\in\Gal(L/K)$ and
$\varphi\in E$, define the ``action''
\begin{equation}
  \label{eq:action}
\sigma.\varphi := \mu_\sigma {}^\sigma\varphi \mu_\sigma^{-1} \in D.
\end{equation}

\begin{lemma}
For every $\sigma\in\Gal(L/K)$ and $\varphi\in E$, $\mu_\sigma {}^\sigma \varphi \mu_\sigma^{-1}\in E$.
\end{lemma}
\begin{proof}
  It is clear from its definition that
  $\mu_\sigma {}^\sigma \varphi \mu_\sigma^{-1}\in D$, we just need to
  prove that it belongs to the center.  Let $\alpha\in D$ and define
  $\beta := \mu_\sigma^{-1}\alpha\mu_\sigma\in \End({}^\sigma
  A)\otimes \QQ$. We have $\beta={}^\sigma({}^{\sigma^{-1}}\beta)$, so
  that
  \begin{align*}
    \mu_\sigma {}^\sigma\varphi \mu_\sigma^{-1}\alpha
    &=\mu_\sigma {}^\sigma\varphi \beta\mu_\sigma^{-1}
      =\mu_\sigma {}^\sigma(\varphi ({}^{\sigma^{-1}}\beta)) \mu_\sigma^{-1}\\
    &=\mu_\sigma {}^\sigma({}^{\sigma^{-1}}\beta \varphi) \mu_\sigma^{-1}
      =\alpha\mu_\sigma {}^\sigma\varphi \mu_\sigma^{-1}.
  \end{align*}
  Therefore $\mu_\sigma {}^\sigma\varphi \mu_\sigma^{-1}\in E$.
\end{proof}

\begin{lemma}
  Under the previous hypothesis, the value $\sigma \cdot \varphi$ does
  not depend on the choice of the isogeny $\mu_\sigma$.
\label{lemma:action-invariance}
\end{lemma}
\begin{proof}
  Let $\widetilde{\mu_\sigma}:{}^\sigma A \to A$ be another
  isogeny. Then
  \[
\widetilde{\mu_\sigma}{}^\sigma\varphi \widetilde{\mu_\sigma}^{-1}= \mu_\sigma (\mu_\sigma^{-1}\widetilde{\mu_\sigma}) {}^\sigma \varphi (\mu_\sigma^{-1} \widetilde{\mu_\sigma})^{-1}\mu_\sigma^{-1}.
\]
The result follows from the fact that
$\mu_\sigma^{-1} \widetilde{\mu_\sigma} \in D$ and $\varphi$ is in its
center.
\end{proof}

\begin{lemma}
  The map $\Gal(L/K) \times E \to E$ defined by (\ref{eq:action})
  gives an action of $\Gal(L/K)$ on $E$.
\end{lemma}
\begin{proof}
  To easy notation, if $\sigma,\tau\in\Gal(L/K)$, we denote by
  $c(\sigma,\tau)=\mu_\sigma {}^\sigma\mu_\tau
  \mu_{\sigma\tau}^{-1}\in D$. Let $\sigma, \tau \in \Gal(L/K)$ and
  $\varphi \in E$. Then
\begin{align*}
  \sigma.(\tau.\varphi) 
  &= \mu_\sigma {}^\sigma(\mu_\tau {}^\tau \varphi \mu_\tau^{-1})\mu_\sigma^{-1}
    = (\mu_\sigma {}^\sigma\mu_\tau) {}^{\sigma\tau}\varphi (\mu_\sigma {}^\sigma\mu_\tau)^{-1}\\
  &= c(\sigma,\tau)\mu_{\sigma\tau} {}^{\sigma\tau}\varphi \mu_{\sigma\tau}^{-1}c(\sigma,\tau)^{-1}
    = \mu_{\sigma\tau} {}^{\sigma\tau}\varphi \mu_{\sigma\tau}^{-1}
    = (\sigma\tau).\varphi.
\end{align*}
In the second-to-last equality, we have used that
$\mu_{\sigma\tau} {}^{\sigma\tau}\varphi \mu_{\sigma\tau}^{-1}\in
E$ (the center), and in particular is an element that commutes with
$c(\sigma,\tau)$. 
\end{proof}
It is easy to verify that
$\sigma \cdot (\varphi + \psi) = \sigma \cdot \varphi + \sigma \cdot
\psi$, that
$\sigma \cdot (\varphi \psi) = (\sigma \cdot \varphi)(\sigma \cdot
\psi)$ and that the action on elements of $\QQ$ is trivial. In
particular, the previous action defines a group homomorphism
\begin{equation}
  \label{eq:morphism}
  \psi: \Gal(L/K) \to \Aut(E).
\end{equation}
Let $G_0$ be the kernel of $\psi$, and let $E_0$ be the subfield of
$E$ fixed by the action of $\Gal(L/K)$.

\begin{proposition}
  The field extension $E/E_0$ is Galois, $G_0$ is a normal
  subgroup of $\Gal(L/K)$, and we have an isomorphism
  \[
    \Gal(L/K)/G_0\simeq \Gal(E/E_0).
  \]
  Moreover, the isomorphism induces an order-reversing bijection
  between subgroups $H$ of $\Gal(L/K)$ containing $G_0$, and
  subextensions of $E/E_0$.
\label{prop:Galois-theory}
\end{proposition} 
\begin{proof}
  The correspondence associates to a subgroup $H$ of $\Gal(L/K)$ the
  field
\begin{equation}
  \label{eq:fixed}
  E^H=\{ \varphi \in E \; : \; \mu_\sigma {}^\sigma\varphi = \varphi \mu_\sigma, \; \forall \sigma \in H\}.
\end{equation}
The proof that the map gives the stated bijection is standard, and
follows the usual correspondence in Galois theory.
\end{proof}

\begin{theorem}
  There exists an intermediate field $K\subseteq L'\subseteq L$ such
  that $L/L'$ and $L'/K$ are Galois extensions, and such that $A$ is a
  strong $L'$-variety.
\label{thm:strong-field}
\end{theorem}
\begin{proof}
  Let $L':=L^{G_0}$.  It is clear that $L/L'$ and $L'/K$ are Galois
  extensions. For each $\sigma\in\Gal(L/L')$ we consider the algebra
  automorphism $\Psi:D\to D$, defined by
  $\Psi(\varphi)=\mu_\sigma {}^\sigma\varphi \mu_\sigma^{-1}$.  Since
  $\Gal(L/L')\simeq G_0$, $\Psi$ is an $E$-algebra homomorphism. Then
  Skolem-Noether's theorem implies the existence of
  $\alpha_\sigma\in D^\times$ such that
  \[
    \mu_\sigma {}^\sigma\varphi \mu_\sigma^{-1} = \alpha_\sigma^{-1}
    \varphi \alpha_\sigma
  \]
  for each $\varphi\in D$. Then $A$ is a strong $L'$-variety with
  respect to the system of isogenies
  $\{\alpha_\sigma \mu_\sigma\}_{\sigma\in G_0}$.
\end{proof}

\begin{corollary}\label{corollary:strong-K-varieties}
  Let $L/K$ be a Galois extension, and let $A/K$ be a
  $K$-variety. Then $A$ is a strong $K$-variety if and only if for all
  elements $\varphi \in E$ relation (\ref{eq:strong-k-var})
  holds.
\end{corollary}
\begin{proof}
  One implication is clear. For the other, suppose that
  (\ref{eq:strong-k-var}) holds for all elements in $E$. Then
  $E = E_0$, so $G_0 = \Gal(L/K)$ and the result follows from last
  theorem.
\end{proof}

Let $L':=L^{G_0}$ be as in Theorem~\ref{thm:strong-field}, so that $A$
is a strong $L'$-variety. Let $\ell$ be a rational prime such that all
prime ideals $\lambda$ of $E$ dividing $\ell$ satisfy Property
(P). Fix one such $\lambda$ and let $W_\lambda$ be the
$E_\lambda$-vector space of dimension $n = \frac{2g}{d[E:\QQ]}$ of
Theorem~\ref{theorem:pairing-existence} and let
$\rho_\lambda:\Gal_L \to \GL(W_\lambda)$ be the irreducible
representation. Let $\sigma \in \Gal(L/K)$, and as
in the previous section, denote by ${}^\sigma\rho_{\lambda}$ the
representation defined by
${}^\sigma\rho_{\lambda}(\tau)=\rho_{\lambda}(\sigma^{-1}\tau\sigma)$.

\begin{lemma}
\label{lemma:extension-K-varieties}
  The representation ${}^{\sigma}\rho_{\lambda}$ is isomorphic to
  $\rho_{\lambda}$ for all $\sigma \in \Gal_{L'}$.
\end{lemma}

\begin{proof}
The fact that $A$ is a strong $L'$-variety implies that if
  $\sigma \in \Gal_{L'}$ then we can assume that $\mu_{\sigma}$ is a
  map of $E_\lambda$-vector spaces, giving an isomorphism between
  $\rho_\lambda$ and ${}^{\sigma}\rho_\lambda$.
\end{proof}

By Theorem~\ref{theorem:extension} there exists a character
$\theta:\Gal_L \to \overline{\QQ}^\times$ such that the twisted
representation $\rho_\lambda \otimes \theta$ can be extended to a
representation $\widetilde{\rho_\lambda}$ of $\Gal_{L'}$.

\begin{remark}
  \label{remark:irreducible-extension}
  It is not hard to prove a converse of
  Lemma~\ref{lemma:extension-K-varieties}, so the representation
  $\rho_\lambda$ cannot be extended any further. However, the
  representation $\widetilde{\rho_\lambda}$ is defined up to a twist,
  and it might be the case that a twist of the representation does
  extend a little further. For example if $E/\QQ$ is an elliptic
  curve, $K$ is a number field, and $\theta$ is a quadratic character
  of $K$ that does not extend to any subfield, then the Galois
  representation attached to $E \otimes \theta$ does not extend to a
  Galois group larger than $\Gal_K$, but a twist of it does. To unify
  statements, we assume from now on that the representation
  $\widetilde{\rho_\lambda}$ is not isomorphic to
  ${}^\sigma\widetilde{\rho_\lambda}$ for any $\sigma \in \Gal(L'/K)$
  different from the identity (we can always take a quadratic twist for this to be true).
\end{remark}

\begin{definition} Let $A/L$ be an abelian $K$-variety.  A Galois
  representation of $\Gal_K$ attached to it is
  $\rho_{A,\lambda}:=\Ind_{\Gal_{L'}}^{\Gal_K}\widetilde{\rho_{\lambda}}$.
\label{definition:Galois-reps}
\end{definition}

Denote by $\mathcal V_\lambda$ the vector space of the induced
representation, namely
\begin{equation}
  \label{eq:direct-sum}
  \mathcal V_\lambda = \bigoplus_{\sigma\in \Gal(L'/K)}\sigma \widetilde{W_\lambda}.  
\end{equation}
  \label{lemma:dimension}

  \begin{theorem}
\label{thm:irreducibility-dimension}
The Galois representation $\rho_{A,\lambda}$ is
absolutely irreducible of dimension $\frac{2g}{d[E_0:\QQ]}$. 
\end{theorem}

\begin{proof}
  The representation $\rho_\lambda$ is irreducible by
  \ref{theorem:pairing-existence}, so the same holds for its extension
  $\widetilde{\rho_\lambda}$ to $\Gal_{L'}$. By
  Remark~\ref{remark:irreducible-extension},
  ${}^\sigma\widetilde{\rho_\lambda}$ is not isomorphic to
  $\widetilde{\rho_\lambda}$ if $\sigma \in \Gal(L'/K)$ is not
  trivial, so irreducibility follows from Mackey's irreducibility
  criterion.

  The dimension of the Galois representation $\rho_{A,\lambda}$ equals
  $\frac{2g}{d[E:\QQ]}[L':K]$, which by
  Proposition~\ref{prop:Galois-theory} equals
  $\frac{2g}{d[E:\QQ]}[E:E_0] = \frac{2g}{d[E_0:\QQ]}$ as claimed.
\end{proof}

\subsection{Dimension}
\label{subsection:dimension}
A priory the dimension of the representation $\rho_{A,\lambda}$
depends on the extension $L/K$. A natural problem is to
understand its dependence, expecting that the
dimension should only decrease when the variety attains extra
endomorphisms. Let $\tilde{L}/L$ be a finite Galois extension, let
$\tilde{D}$ denote the endomorphism algebra of $A/\tilde{L}$ (tensored
with $\QQ$), let $\widetilde{E}$ denote its center, and let
$\widetilde{E}_0$ denote the subfield of $\widetilde{E}$ fixed by
$\Gal(\tilde{L}/K)$. Let $d$ and $\tilde d$ be the Schur indices of
$D$ and $\tilde D$, respectively. By Theorem~\ref{theorem:albert},
these indices can a priory be any positive integer. A short exercise
(or Lemma~\ref{lemma:degree E0 tildeE0} below) shows that $d\mid \tilde d$
whenever $A_{\tilde L}$ is simple.

\begin{lemma}\label{lemma:degree E0 tildeE0}
  We have the inclusion $\widetilde{E}_0 \subseteq E_0$. If
  $A_{\tilde L}$ is simple, then
  $[E_0:\widetilde{E}_0] \mid \frac{\tilde{d}}{d}$.
\end{lemma}
\begin{proof}
  For the inclusion, note that
  \[
    \widetilde{E}_0=\widetilde{E}^{\Gal(\widetilde{L}/K)} =
    (\widetilde{E}^{\Gal(\widetilde{L}/L)})^{\Gal(L/K)} =
    (\widetilde{E}\cap E)^{\Gal(\widetilde{L}/K)}=\widetilde{E}_0\cap
    E_0.
  \]
	
  \noindent
  \begin{minipage}{0.7\textwidth}
    \setlength{\parindent}{12.0pt}
    To bound the degree
    $[E_0:\widetilde{E}_0]$, the reader might find the displayed field
    relations useful. First, we bound $[E:\widetilde{E}\cap
    E]$. Consider the compositum $E\tilde E$ in $\tilde D$. By
    applying the Grunwald-Wang theorem \cite[\S~18.6]{MR0674652} (see
    e.g. \cite[Lemma~7.4.1]{MR3904148}), $D$ contains a maximal
    subfield $M$ which is $E$-linearly disjoint from $E\tilde
    E$. Since $[M:E]=d$, we have $[M\tilde E:E\tilde E]=d$. The degree
    $[M\tilde E:\tilde E]$ divides $\tilde{d}$, therefore,
    $[E\tilde E:\tilde E]$ divides $\frac{\tilde{d}}{d}$.
	
    Now, the extension $\widetilde{E}/\widetilde{E}_0$ is Galois,
    therefore, the inclusion
    $\widetilde{E}_0\subseteq E \cap \widetilde{E}$ implies we have an
    equality $[\widetilde{E}:\widetilde{E}\cap E]=[E\widetilde{E}:E]$,
    and so $[E\widetilde{E}:\widetilde{E}]=[E:E \cap
    \widetilde{E}]$. It follows that $[E:E \cap \widetilde{E}]$
    divides $\frac{\tilde{d}}{d}$. 
    
    Apply now a similar argument to the extension
    $[E_0:\widetilde{E}_0]$. We have seen above that
    $\widetilde{E}_0 = (E\cap \widetilde E)^{\Gal(\widetilde L/K)}$,
    so $(E\cap \widetilde E) / \widetilde{E}_0$ is Galois, and
    therefore 
    $[E:E_0] = [E\cap\widetilde E:\widetilde{E}_0]$ and
    $[E:E\cap\widetilde E]= [E_0:\widetilde{E}_0]$. Hence
    $[E_0:\widetilde{E}_0]$ divides $\frac{\tilde d}{d}$.
  \end{minipage}
  \hfill\begin{minipage}{0.30\textwidth}
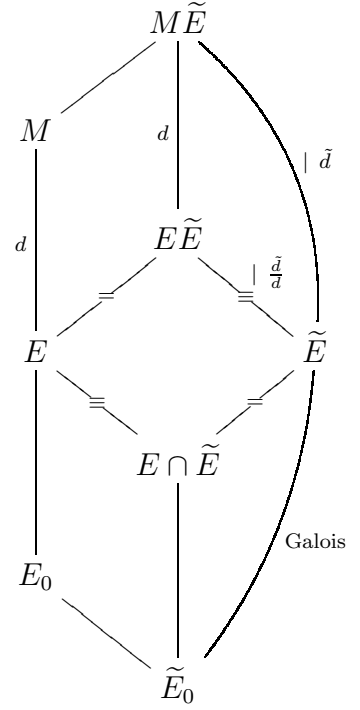
\[
\xymatrix{
& M\widetilde{E} \ar@{-}[dl] \ar@{-}[dd]_{d} \ar@/^1.6pc/@{-}[dddr]^{\mid\ \tilde d}\\
M \ar@{-}[dd]_{d}\\
& E\widetilde{E} \ar@{-}[dl]|-{=} \ar@{-}[dr]^{\mid\ \frac{\tilde d}{d}}|-{\equiv}\\
E \ar@{-}[dr]|-{\equiv} \ar@{-}[dd] && \widetilde{E} \ar@{-}[dl]|-{=} \ar@/^1pc/@{-}[dddl]^{\text{Galois}}\\
& E\cap \widetilde{E} \ar@{-}[dd]\\
E_0 \ar@{-}[dr]\\
& {\widetilde{E}_0}
}
\]
\end{minipage}
\end{proof}

Let $\widetilde{\sigma_{A,\lambda}}$ be the Galois representation of
$\Gal_K$ attached to $A/\tilde{L}$ as in
Definition~\ref{definition:Galois-reps}.

\begin{theorem}
  If $A_{\tilde L}$ is simple, then
  \(
    \dim(\widetilde{\sigma_{A,\lambda}})\mid \dim(\rho_{A,\lambda}).
  \)
  Moreover, equality holds if and only if
  $[E_0:\widetilde{E}_0]=\tilde d/d$, and in particular whenever $\tilde d = d$.
\label{thm:dimension}
\end{theorem}
\begin{proof}
  By Lemma~\ref{lemma:dimension} we have
  $\dim(\rhoA)=\frac{2g}{d[E_0:\QQ]}$ and
  $\dim(\rhoAt)=\frac{2g}{\tilde{d}[\widetilde{E}_0:\QQ]}$. Hence
  the divisibility between dimensions holds by Lemma~\ref{lemma:degree
    E0 tildeE0}. For the second statement, we need three facts.

\vspace{2pt}
\noindent {\bf Fact 1:}
$\widetilde{E}^{\Gal(\widetilde{L}/L)}=\widetilde{E}\cap
D=\widetilde{E}\cap E$.
\vspace{2pt}

Since the variety $A/L$ is isogenous to all of its Galois conjugates,
we can assume that $\mu_{\sigma}$ is the identity for all
$\sigma \in \Gal(\widetilde{L}/L)$ (recall that the action of
$\Gal(\widetilde{L}/K)$ on $\widetilde{E}$ is independent of the
choice, by Lemma~\ref{lemma:action-invariance}). Then if
$\sigma \in \Gal(\widetilde{L}/K)$ and $\varphi \in \widetilde{E}$,
\[
\sigma \cdot \varphi = {}^\sigma \varphi,
\]
i.e. the action is given by the Galois action on the coefficients of
the morphism $\varphi$. By Galois theory,
$\varphi \in \widetilde{E}^{\Gal(\widetilde{L}/L)}$ is and only if the
coefficients of the morphism can be chosen to be in $L$, i.e. if
$\varphi \in D$, hence the first equality. The second follows from the
fact that an element in the center of $\widetilde{D}$ is also in the
center of $D$ (because $D \subset \widetilde{D}$).

\vspace{2pt}
\noindent{\bf Fact 2:} if $E \subset \widetilde{E}$ then $\dim(\rhoA) = \dim(\rhoAt) \frac{\tilde d}{d}$.
\vspace{2pt}

After choosing an extension of each element in $\Gal(L/K)$ to
$\widetilde{L}$ (any will do),
$\Gal(\widetilde{L}/K)=\Gal(\widetilde{L}/L)\Gal(L/K)$, hence
\[
\widetilde{E}_0=\widetilde{E}^{\Gal(\widetilde{L}/K)}=(\widetilde{E}^{\Gal(\widetilde{L}/L)})^{\Gal(L/K)} = E^{\Gal(L/K)}=E_0,
\]
where the middle equality comes from Fact 1 (since
$\widetilde{E}\cap E=E$ under our hypothesis). The statement follows
from Lemma~\ref{lemma:dimension}.

\vspace{2pt}
\noindent{\bf Fact 3:} If $\tilde d = d$, then $E\subseteq \tilde E$.
\vspace{2pt}

This is \cite[Lemma~7.4.2]{MR3904148}, but the proof is much simpler
for division algebras. Consider the subalgebra
$D\otimes_E E\tilde E\subset \tilde D$ obtained by extending scalars
to the field compositum $E\tilde E$, and note that
\[
	\dim_{\tilde E} D\otimes_E E\tilde E =\dim_E D\cdot [E\tilde E:\tilde E] = d^2[E\tilde E:\tilde E] \leq \dim_{\tilde E}\tilde D=\tilde d^2.
\]
The equality $d=\tilde d$ implies $[E\tilde E:\tilde E]=1$, hence
$E\subseteq \tilde E$.

The previous three facts imply that $\dim(\rhoA) = \dim(\rhoAt)$ if
and only if either $d=\tilde d$ or $[E_0:\widetilde{E}_0]=\tilde d/d$:
suppose that the equality of dimensions holds, then 
$[E_0:\widetilde{E}_0]=\tilde d/d$.

Conversely, if $[E_0:\widetilde{E}_0]=\tilde d/d$, then comparing the
dimensions again yields the equality. Finally, suppose that
$\tilde d=d$. Then by Fact 3 we have $E\subseteq \tilde E$, and Fact 2
shows $\dim(\rhoA)=\dim(\rhoAt)$.
\end{proof}

\begin{theorem}
  Suppose that $d$ and $\tilde d\leq 2$.
  Keeping the previous notation, we get the following relation between
  $\dim(\rhoA)$ and $\dim(\rhoAt)$:
  \begin{itemize}
  \item They are the same if either: $\dim_{\tilde E}\tilde{D} = \dim_E D$ or $[E_0:\widetilde{E}_0]=2$,
  \item $\dim(\rhoA)=2\dim(\rhoAt)$ if
    $\dim_{\tilde E}\tilde{D}\neq \dim_E D$ and either
    $E \subset \widetilde{E}$ or $E_0 = \widetilde{E}_0$.
  \end{itemize}
\label{thm:dimension2}
\end{theorem}

\begin{proof}
  If $D$ and $\widetilde{D}$ have the same type, then
  $E=Z(D)\subset Z(\widetilde{D})=\widetilde{E}$ and
  $d = \tilde d$ so the statement follows from Fact
  2 (of the previous theorem). Otherwise, it must be the case that $\widetilde{D}$ is a
  quaternion algebra, while $D$ is abelian (in which case
  $\tilde d=2$ and $d=1$). If $E \subset \widetilde{E}$
  then the statement follows once again from Fact 2. Suppose then that
  $E \cap \widetilde{E} \subsetneq E$ (so
  $[E : E \cap \widetilde{E}]=2$). Then
\[
  [E_0:E_0 \cap \widetilde{E}_0] = [E^{\Gal(\widetilde{L}/K)}:E^{\Gal(\widetilde{L}/K)} \cap \widetilde{E}^{\Gal(\widetilde{L}/K)}] \le [E:E \cap \widetilde{E}]=2.
\]
On the other hand,
\[
\widetilde{E}_0=\widetilde{E}^{\Gal(\widetilde{L}/K)} = (\widetilde{E}^{\Gal(\widetilde{L}/L)})^{\Gal(L/K)} = (\widetilde{E}\cap E)^{\Gal(\widetilde{L}/K)}=\widetilde{E}_0\cap E_0.
\]
In particular, $[E_0:\widetilde{E}_0] \le 2$. If it equals $1$, then
$\dim(\rho_{A,\lambda})=2\dim(\widetilde{\sigma_{A,\lambda}})$ by
Lemma~\ref{lemma:dimension}, while if it equals $2$, then both
dimensions are the same as claimed.
\end{proof}

\section{Equivariant pairings}
\label{section:pairings}
It is natural to wonder whether/when
the constructed representation preserves a bilinear form.
Let $K$ be a field. A subgroup $G\subseteq \GL_n(K)$ is
called irreducible if there are no non-trivial proper $G$-invariant
subspaces of $K^n$. We have the following variant of Schur's lemma.

\begin{proposition}\label{lemma:pairing-schur}
  Let $G$ be an irreducible group, and let $\chi:G\to K^\times$,
  ${}^\dagger:G\to GL_n(K)$ be any two maps. Let $M,N\in \GL_n(K)$ be
  matrices satisfying
  \[
    \begin{cases}
      B^\dagger MB &= \chi(B) M\\
      B^\dagger NB &= \chi(B) N.
    \end{cases}
  \]
  for all $B\in G$.  Then, there exists a unique $\lambda\in K^\times$
  such that $M=\lambda N$.
\end{proposition}
\begin{proof}
  Let $p(x)=\det(M-xN)\in K[x]$, a polynomial of degree $n$ and
  leading coefficient $(-1)^n\det N\neq 0$. Let
  $\lambda\in\overline{K}$ be a root of $p(x)$. Clearly
  $\lambda\neq 0$, since $M$ would be singular otherwise. Let
  $v\in\ker(M-\lambda N)$ be any nonzero vector. Then, for every
  $B\in G$ we have
  \[
    (M-\lambda N)Bv = \chi(B) (B^\dagger)^{-1} (M-\lambda N)v = 0,
  \]
  hence $G$ stabilizes $\ker(M-\lambda N)$. But this is a nontrivial
  subspace of $K^n$, and therefore (since $G$ is irreducible)
  $\ker(M-\lambda N)=K^n$, so $M=\lambda N$ and $\lambda\in
  K^\times$. Uniqueness follows from the equality 
  $p(x)=(\lambda-x)^n$.
\end{proof}

The last proposition implies that once a similitude
character and a type are fixed, there exists at most one 
$G$-equivariant pairing with the given type and similitude character.

\begin{corollary}
  Let $G$ be an irreducible subgroup of $\GL_n(K)$, $V=K^n$, and let
  $\Psi_i:V\times V\to K$, $i=1,2$ be two biadditive form. Suppose
  that there exists a character $\chi:G\to K^\times$ such that, for
  every $g\in G$ and all $v,w\in V$,
  \[
    \Psi_i(gv,gw) = \chi(g)\Psi_i(v,w), \text{ for }i=1,2.
  \]
  If both $\Psi_1$ and $\Psi_2$ are alternating (resp. symmetric, or
  hermitian), then there exists $\lambda\in K^\times$ such that
  $\Psi_1=\lambda\Psi_2$.
\end{corollary}
\begin{proof} Follows from the proposition.
\end{proof}

\begin{proposition}\label{proposition:pairing-extension}
  Let $G$ be a subgroup of $\GL_n(K)$ and let $H$ be a normal
  irreducible subgroup of $G$. Let $V={K}^n$. Suppose that
  there exists a biadditive $H$-invariant form $\Psi:V\times V\to K$
  with similitude character $\chi:H\to K^\times$, i.e. for every
  $h\in H$ and all $v,w\in V$,
  \[
    \Psi(hv,hw) = \chi(h)\Psi(v,w).
  \]
  Suppose that $\Psi$ is alternating, symmetric or hermitian.  Then
  the following are equivalent:
  \begin{enumerate}
  \item The form $\Psi$ is $G$-invariant (for some similitude character).
    
  \item The character $\chi$ extends to a character
    $\widetilde{\chi}:G \to {K}^\times$.
    
  \item For all $g \in G$, $h \in H$, $\chi(ghg^{-1})=\chi(h)$.
  \end{enumerate}
\end{proposition}
\begin{proof} It is clear that $(1) \Rightarrow (2) \Rightarrow (3)$
  (the character $\widetilde{\chi}$ being the similitude character).
  Suppose that the biadditive form $\Psi$ is alternating or
  symmetric. Then, there exists an antisymmetric matrix $J\in\GL_n(K)$
  such that
  \[
    \Psi(v,w)=v^\top J w.
  \]
  The invariance property translates into $h^\top J h = \chi(h)J$ for
  all $h\in H$. Since $H$ is normal in $G$, for all $g\in G$ and all
  $h\in H$ we have $ghg^{-1}\in H$, and hence
  \[
    (ghg^{-1})^\top J (ghg^{-1}) = \chi(ghg^{-1})J=\chi(h)J.
  \]
  By direct manipulation we obtain the equality
  \( h^\top (g^\top J g)h = \chi(h) g^\top Jg \).
  Proposition~\ref{lemma:pairing-schur} (with $h^\dagger = h^\top$,
  $M=g^\top Jg$ and $N=J$) implies that there is a unique
  $\tilde \chi(g)\in K^\times$ such that
  $$g^\top Jg = \tilde \chi(g)J.$$
  The similitude map $g\mapsto \tilde\chi(g)$ is actually a character, since
  $(gg')^\top J gg' = \tilde\chi(gg')J =
  \tilde\chi(g)\tilde\chi(g')J$. By uniqueness of each
  $\tilde\chi(g)$, we obtain $\tilde\chi|_H=\chi$.
	
  The hermitian case works the same replacing $g^\top$ by
  $\overline{g}^\top$, where $\bar\cdot$ is the involution of $K$.
\end{proof}

\begin{remark}
The proposition implies that the representation
$\widetilde{\rho_\lambda}$ attached to a strong $K$-variety preserves
the bilinear form  $\Psi$  from Theorem~\ref{theorem:pairing-existence}
precisely when the twisting character $\psi$ 
(as in Theorem~\ref{theorem:extension}) satisfies that
$\psi^2(ghg^{-1})=\psi^2(h)$ for all $h \in \Gal_L$, $k \in \Gal_K$.

The same remark implies that when $L/K$ is cyclic, $\psi$ can be taken to be trivial,
so in this case the representation always preserves a bilinear form.
\label{remark:strong-bilinear}
\end{remark}

A similar result holds while studying existence of an
invariant bilinear form for the induction of a representation.

\begin{proposition}
  Let $G$ be a group and $H\vartriangleleft G$ be a finite index
  normal subgroup. Let $K$ be a field and let $W$ be a finite
  dimensional $K$-vector space. Let $\rho:H \to \GL(W)$ be a
  representation and let $\Psi:W \times W \to K$ be an additive
  $H$-invariant form with similitude character $\chi$. Let
  \begin{equation}
    \label{eq:induction-space}
V = \sum_{h \in G/H} hW,    
  \end{equation}
be the underlying $K$-vector space of the induced representation
$\Ind_H^G \rho$.  Then the following are equivalent:
\begin{enumerate}
\item There exists a $K$-bilinear, $G$-equivariant non-degenerate
  pairing $\widetilde{\Psi}$ on $V$ whose restriction to $W$ equals $\Psi$.
\item The character $\chi$ extends to a character of $G$.
\end{enumerate}
\label{prop:extension-induction}  
\end{proposition}
\begin{proof}
  Start supposing that there exists a $G$-equivariant pairing
  $\widetilde{\Psi}:V\times V\to K$ with similitude character $\widetilde{\chi}$
  such that $\widetilde{\Psi}|_{W \times W} = \Psi$ (identifying $W$ with the
  susbpace corresponding to $h=1$ in \eqref{eq:induction-space}). The
  equivariance property implies that, for all
  $u,v\in W$ and all $g\in G$,
  \[
    \widetilde{\Psi}(g \cdot u,g \cdot v) =
    \widetilde{\chi}(g)\widetilde{\Psi}(u,v) = \widetilde{\chi}(g)\Psi(u,v).
  \]
  If $g \in H$ then the left hand side also equals $\chi(g)\Psi(u,v)$
  so (since $\Psi$ is non-degenerate)
  $\chi(g) = \widetilde{\chi}(g)$ for $g \in H$; $\widetilde{\chi}$
  is the desired extension.
  
  Conversely, suppose that $\chi$ extends to a character
  $\widetilde{\chi}$ of $G$. Let $\{h_1,\ldots,h_n\}$ be
  representatives for $G/H$ with $h_1=1$. Define a bilinear pairing
  $\widetilde{\Psi}:V\times V\to K$ by decreeing that:
  \begin{enumerate}
  \item $\widetilde{\Psi}(h_iW,h_jW) = 0$ if $i \neq j$,
    
  \item $\widetilde{\Psi}(h_i u,h_i v) := \widetilde{\chi}(h_i) \Psi(u,v)$ for $i=1,\ldots,n$.
  \end{enumerate}
  The pairing $\widetilde{\Psi}$ thus defined satisfies the expected properties,
  namely: it is $G$-equivariant, non-degenerate and its restriction to
  $W\times W$ coincides with $\Psi$.

\end{proof}

\begin{remark}
  If the similitude character of $\chi$ extends to $G$, its extension
  need not be unique. There are as many extensions as characters of
  the group $G/H$, and each extension gives a different pairing.
\end{remark}

Propositions~\ref{proposition:pairing-extension} and
\ref{prop:extension-induction} imply that the representation
$\rho_{A,\lambda}$ preserves an extension of the pairing obtained in
Theorem~\ref{theorem:pairing-existence} if and only if the character
$\psi$ of Theorem~\ref{theorem:extension} (or a multiple of it by the square of a character of $\Gal_{L'}$) extends to $\Gal_K$.

\section{Applications}

\subsection{Some atypical images of Galois representations}

Let $A$ be an abelian fourfold defined over a number field $K$. Let
$L/K$ be a finite Galois extension such that
$\End^0(A_{\overline{K}}) = \End^0(A_L)$. Assume the following two
conditions:
\begin{enumerate}
\item $F=\End^0(A)$ is an imaginary quadratic field, and
\item $\End^0(A_L)$ is a (possibly split) indefinite quaternion
  algebra over $\QQ$.
\end{enumerate}
Let $\ell$ be a rational prime that remains inert in $F$ 
and let $\lambda$ be the prime of $F$ over $\ell$, so that $F_\lambda$
is a quadratic extension of $\QQ_\ell$. Let $\gamma$ be the nontrivial
automorphism of $F_\lambda/\QQ_\ell$. The space $V_\ell(A)$ has an
$F_\lambda$-module structure; denote by $W_\lambda$ the vector space
$V_\ell(A)$ with this structure. By
Theorem~\ref{theorem:pairing-existence}, $W_\lambda$ is irreducible as
an $F_\lambda[\Gal_K]$-module and has $F_\lambda$-dimension 4. Denote
by $\rho_{A,\lambda}$ the corresponding $\lambda$-adic
representation.

The space $W_\lambda$ comes with a non-degenerate $F_\lambda$-hermitian
pairing $\Psi_\lambda:W_\lambda\times W_\lambda\to F_\lambda$, with
the property that for all $v,w\in W_\lambda$ and all
$\sigma\in \Gal_K$,
\[
  \Psi_\lambda(\rho_{A,\lambda}(\sigma)v,\rho_{A,\lambda}(\sigma)w) =
  \chi_\ell(\sigma)\cdot\Psi_\lambda(v,w).
\]
Hence the image of $\rho_{A,\lambda}$ is contained in
$\operatorname{GU}_4(\QQ_\ell)$. 

On the other hand, the same theorem can be applied to $A_L$, but now the
center of $\End^0(A_L)$ is $\QQ$. Hence there exists an absolutely
irreducible $\QQ_\ell[\Gal_L]$-submodule $\mathcal W_\ell$ of
$V_\ell(A_L)$, such that
$V_\ell(A_L)\simeq \mathcal W_\ell \times \mathcal W_\ell$, together
with a non-degenerate $\QQ_\ell$-bilinear alternating pairing
$\Phi_\ell: \mathcal W_\ell\times \mathcal W_\ell\to \QQ_\ell$. This
pairing satisfies that, for all $v,w\in \mathcal W_\lambda$ and each
$\tau\in \Gal_L$,
\[
  \Phi_\ell(\rho_{A,\lambda}(\tau)v,\rho_{A,\lambda}(\tau)w) =
  \chi_\ell(\tau)\cdot\Phi_\ell(v,w).
\]
Now there is an isomorphism of $F_\lambda[\Gal_L]$-modules
$W_\lambda\simeq \mathcal W_\ell \otimes_{\QQ_\ell} F_\lambda$. To see this, note that $W_\lambda$ (seen as a $\QQ_\ell[\Gal_L]$-module) is $V_\ell(A_L)$, and hence tensoring with $F_\lambda$ we obtain 
\[
  V_\ell(A_L)\otimes_{\QQ_\ell}F_\lambda \simeq W_\lambda \times
  {}^\gamma W_\lambda.
\]
Here ${}^\gamma W_\lambda$ is $W_\lambda$ with a $\gamma$-linear
action of $F_\lambda$. On the other hand we have
$V_\ell(A_L)\otimes_{\QQ_\ell} F_\lambda \simeq (\mathcal W_\ell
\times \mathcal W_\ell)\otimes_{\QQ_\ell} F_\lambda$, and since all
the modules are absolutely irreducible, we obtain the isomorphism
$W_\lambda\simeq \mathcal W_\ell\otimes_{\QQ_\ell} F_\lambda$.
Hence $W_\lambda$ inherits (by extension of scalars to $F_\lambda$)
the alternating pairing $\Phi_\ell$. By
Proposition~\ref{proposition:pairing-extension} the pairing
$\Phi_\ell$ is also $\Gal_K$-equivariant up to a similitude
character. More explicitly, there exists a finite order character
$\theta:\Gal_K\to \bar\QQ^\times$ which factors through $\Gal(L/K)$,
and such that for all $v,w\in W_\lambda$ and all $\sigma\in \Gal_K$,
we have
\[
  \Phi_\ell(\rho_{A,\lambda}(\sigma)v,\rho_{A,\lambda}(\sigma)w) =
  \theta(\sigma)\chi_\ell(\sigma)\cdot\Phi_\ell(v,w).
\]
This shows that the image of $\rho_{A,\lambda}$ is contained in
$\GSp_4(F_\lambda)$ and has similitude character $\theta\chi_\ell$. We
have thus seen that
\( \rho_{A,\lambda}(\Gal_L) \subseteq
\operatorname{GU}_4(\QQ_\ell)\cap \GSp_4(F_\lambda)\), the
intersection taking place in $\GL_4(F_\lambda)$. Particular examples can be found in the forthcoming preprint \cite{FFG25}.

Observe that $A_L$ is a strong abelian $K$-variety, by
Corollary~\ref{corollary:strong-K-varieties} and the fact that
$Z(\End^0(A_L)) = \QQ\subset F = \End^0(A)$. Therefore the
representation $\rho_{A,\lambda}|_{\Gal_L}$ extends (after a twist) to
a representation of $\Gal_K$ by Theorem~\ref{theorem:extension}, but
the twist is in fact trivial since $\rho_{A,\lambda}$ is already a
representation of $\Gal_K$.

If $\End^0(A_L)$ is a \emph{definite} quaternion algebra (instead of
an indefinite one), the same argument proves that
\( \rho_{A,\lambda}(\Gal_L) \subseteq
\GU_4(\QQ_\ell)\cap \GO_4(F_\lambda) \). An example of
this case is given in \cite{cantoralfarfn2023monodromy}, where $A$ is
the Jacobian of the genus 4 curve
\[
	C: y^2=x(x^4+1)(x^4+x^2+1).
\]

\subsection{Abelian surfaces with potential QM are modular}

Let $A/\QQ$ be an abelian surface such that $\End^0(A_{\overline{\QQ}})$ is
an indefinite (possibly split) quaternion algebra. If $\End^0(A)$ is a
quadratic field, the surface $A$ is modular by Ribet's results
(\cite{MR2058653}), together with Serre's modularity conjecture
(\cite{MR2551763,MR2551764}. We assume that $\End^0(A)=\QQ$, and we claim that $A$ is also modular in this case.
\begin{lemma}\label{lemma:pQM-surface-is-Q-variety}
$A_L$ is a strong $\QQ$-variety.
\end{lemma}
\begin{proof}
  The relation ${}^\sigma\varphi = \varphi$ holds trivially for all
  $\sigma\in\Gal(L/\QQ)$ and all $\varphi\in Z(\End^0(A_L))=\QQ$. Then
  Corollary~\ref{corollary:strong-K-varieties} implies there are
  isogenies $\mu_\sigma:A_L\to A_L$ such that $A_L$ is a strong
  $\QQ$-variety.
\end{proof}

By \cite[Proposition~2.1]{MR2091964}, there exists a smallest Galois
extension $L/\QQ$ such that $D:=\End^0(A_L)=\End^0(A_{\bar\QQ})$, and
such that $\Gal(L/\QQ)$ is either $C_n$ or $D_n$ with $n=2,3,4$ or
$6$.

\begin{lemma}\label{lemma:quadratic-field}
  For each field $K\subseteq L$ such that $L/K$ is cyclic,
  $\End^0(A_K)$ contains a quadratic field. In particular, the
  assumption $\End^0(A)=\QQ$ implies that $L/\QQ$ is a dihedral
  extension, and there exists a quadratic extension $K/\QQ$ such that
  $\End^0(A_K)$ is a quadratic field.
\end{lemma}
\begin{proof}
  By the Skolem-Noether theorem, for each $\sigma\in\Gal(L/\QQ)$ there
  exists some $\alpha_\sigma\in D^\times$ such that
  \( {}^\sigma\varphi = \alpha_\sigma \varphi \alpha_\sigma^{-1} \) for each
  $\varphi\in D$. In particular, for each $\sigma\in\Gal(L/\QQ)$ the
  subalgebra $\End^0(A_{L^{\langle\sigma\rangle}})$ equals the
  centralizer in $D$ of the field $\QQ(\alpha_\sigma)$.

  With this in mind, we let $\sigma\in\Gal(L/\QQ)$ be any element and
  let $K:=L^{\langle\sigma\rangle}$. If $\alpha_\sigma\in\QQ^\times$, then
  the centralizer of $\QQ$ in $D$ is $D$, and in fact
  $K=L$. Otherwise, $\alpha_\sigma$ generates a (maximal) quadratic
  subfield of $D$, and the centralizer of $\QQ(\alpha_\sigma)$ is
  $\QQ(\alpha_\sigma)$ itself. This proves the claim.
\end{proof}

\begin{theorem}\label{theorem:potentialQM-implies-paramodular}
  The abelian surface $A$ is Siegel modular, i.e. there exists a Siegel
  newform of $\GSp_4(\A_\QQ)$ whose $L$-series matches
  that of $A$.
\end{theorem}
\begin{proof}
  Since all endormorphisms of $A$ are defined over $L$, there exists
  an irreducible Galois representation
  $\rho_{\lambda,L}:\Gal_L \to \GL_2(\overline{\QQ_\ell})$ attached to
  $A/L$. By Lemma~\ref{lemma:pQM-surface-is-Q-variety} $A/L$ is a
  strong $\QQ$-variety, so Theorem~\ref{theorem:extension} implies
  that there exists a character $\theta$ of $\Gal_L$ and a
  representation
  $\widetilde{\rho_\lambda}:\Gal_\QQ\to\GL_2(\overline{\QQ_\ell})$
  (with Hodge-Tate weights $\{0,1\}$) such that
  $\widetilde{\rho_\lambda}|_{\Gal_L}\simeq
  (\rho_{\lambda,L})\otimes\theta$.

  Let $\ell$ be a prime number such that the residual representation
  $\overline{\widetilde{\rho_\lambda}}$ is absolutely irreducible
  (this is always the case if $\ell$ is large enough). Serre's
  modularity conjecture (\cite{MR2551763,MR2551764}) implies that the
  residual representation $\overline{\widetilde{\rho_\lambda}}$ is
  modular so (by \cite{MR2600871}) $\widetilde{\rho_\lambda}$ itself
  is modular.  By solvable base change for $\GL_2$ (as proved in
  \cite{MR0574808}), the representation $\widetilde{\rho_\lambda}$
  restricted to $\Gal_L$ is modular, and then the same is true for
  $\rho_\lambda$ (as twisting preserves modularity).
  
  Fix a quadratic extension $K/\QQ$ as given in
  Lemma~\ref{lemma:quadratic-field} and let $E= \End^0(A_K)$ be the
  corresponding quadratic field. Let $\ell$ be a rational prime
  splitting in $E$. Consider the $\ell$-adic representation of $A$,
  $\rho_\ell:\Gal_\QQ\to \GSp_4(\QQ_\ell)$. Fix $\lambda$ a prime of
  $E$ over $\ell$, then there is a subrepresentation (abusing
  notation) $\rho_\lambda:\Gal_K\to\GL_2(E_\lambda)$ of
  $\rho_\ell|_{\Gal_K}$, and
  $\rho_\ell = \Ind^{\Gal_K}_{\Gal_\QQ}\rho_\lambda$. The
  representation $\rho_\lambda$ must be a twist of
  $\widetilde{\rho_\lambda}$ restricted to $\Gal_K$ (as they have the
  same projective representation), so $\rho_\lambda$ is also modular.
  Automorphic induction (as proven in \cite[Th\'eor\`eme
  3]{MR2903769}) then implies that
  $\Ind^{\Gal_K}_{\Gal_\QQ}\rho'_\lambda \simeq \rho_\ell$ is modular
  itself.

We are led to prove that the automorphic representation $\Pi$ of
$\GL_4(\A_\QQ)$ is actually a transfer from $\GSp_4(\A_\QQ)$. The fact
that $A$ preserves the Weil pairing implies that $\Pi$ has symplectic
type (as in \cite{MR4349242} \S 2.9) and the result then follows from
\cite[Theorem 2.9.3]{MR4349242}.
\end{proof}
  In general we do not know how to prove that the automorphic form of
  the last theorem has a paramodular fixed vector. Following the
  previous notation, if $\rho_\lambda$ has trivial Nebentypus
  (i.e. its similitude character equals the cyclotomic one), then we can show the paramodularity of its induction to $\Gal_\QQ$. We do that as follows.

\begin{proposition}\label{proposition:real-quadratic-field}
  Suppose $\Gal(L/\QQ)\simeq D_4 \simeq C_2\times C_2$. Then, there is
  at most one quadratic subfield $K$ of $L$ such that $E=\End^0(A_K)$
  is an imaginary quadratic field. If such $K$ exists, the determinant
  of the representation $\rho_\lambda:\Gal_K\to \GL_2(E_\lambda)$
  equals $\eta_{L/K}\cdot\chi_\ell$, with $\eta_{L/K}$ the nontrivial
  character of $\Gal(L/K)$. For any other quadratic
  subfield $K'\subset L$, $\End^0(A_{K'})$ is real quadratic, and the
  determinant of the corresponding representation is $\chi_\ell$.
\end{proposition}
\begin{proof}
  Let $K_1$ and $K_2$ be any two different quadratic subfields of
  $L$. We will show that either $E_1=\End^0(A_{K_1})$ or
  $E_2=\End^0(A_{K_2})$ is a real quadratic field.

  Suppose for a contradiction that $E_1$ and $E_2$ are both imaginary
  quadratic fields. We first show that $E_1\neq E_2$. For that, let
  $\sigma_i\in \Gal(L/\QQ)$ be the nontrivial automorphism fixing
  $K_i$ for $i=1,2$. Then
  $E_i=\End^0(A_{L^{\langle \sigma_i\rangle}})$ for $i=1,2$ by
  Lemma~\ref{lemma:quadratic-field}. Hence if $E_1= E_2$, we would
  have
  $E_1\subset \End^0(A_{L^{\langle\sigma_1,\sigma_2\rangle}})=\End^0(A)$. But
  this is a contradiction, since we are assuming that
  $\End^0(A)=\QQ$. Hence $E_1\neq E_2$.

  It follows that the quaternion algebra $D=\End^0(A_L)$ is generated
  (as an algebra over $\QQ$) by two quadratic elements $\sqrt{a_1}$,
  $\sqrt{a_2}$ with $a_1,a_2\in\QQ$ that generate $E_1$ and $E_2$,
  respectively. These elements do not commute, since $E_1\neq E_2$ are
  maximal fields and $D$ is noncommutative. Since $\End^0(A_L)$ is an
  indefinite quaternion algebra, either $a_1$ or $a_2$ must be
  positive: otherwise, $D\otimes_\QQ \mathbb R$ would be isomorphic to the
  Hamilton quaternions, which contradicts the fact that $D$ is indefinite.  Therefore either
  $E_1=\QQ(\sqrt{a_1})$ or $E_2=\QQ(\sqrt{a_2})$ is a real quadratic
  field.

  Let $K\subset L$ be any quadratic subfield and let $E=\End^0(A_K)$
  be the corresponding quadratic field of endomorphisms. Given
  $\sigma\in\Gal_K$, by Skolem-Noether we have that there exists some
  $\alpha_\sigma\in E^\times$ such that
  \( {}^\sigma \varphi = \alpha_\sigma\varphi \alpha_\sigma^{-1} \)
  for all $\varphi\in D$. Each $\alpha_\sigma$ is determined up to
  multiplication by elements in $\QQ^\times$, the center of $D$. It
  can be checked that the determinant of the representation
  $\rho_\lambda:\Gal_K\to\GL_2(E_\lambda)$ is the cyclotomic character
  times
\begin{align*}
	\chi:\Gal(L/K)&\to E^\times\\
	\sigma &\mapsto \alpha_\sigma/\overline{\alpha_\sigma}
\end{align*}
where $\bar\cdot$ denotes complex conjugation (cf. \cite{MR2058652},
Lemma 5.11).  In particular, when $E$ is a real field, then complex
conjugation restricts to the identity on $E$, and so $\chi$ is
trivial.

When $E$ is an imaginary field, we use
\cite[Proposition~3.5]{MR2058653} to find a trace of Frobenius
$a_{\mathfrak p}\neq 0$ with $E=\QQ(a_{\mathfrak p})$, and then by
\cite[Proposition~3.4]{MR2058653} we have
$a_{\mathfrak p} = \chi(\mathrm{Frob}_{\mathfrak p})
\overline{a_{\mathfrak p}}$. This shows
$\chi(\mathrm{Frob}_{\mathfrak p})=-1$, and so $\chi$ is nontrivial
when $E$ is imaginary quadratic.
\end{proof}

\begin{remark}
  Keeping the last proposition hypothesis, if there exists a quadratic
  field $K$ such that $\End^0(A_K)$ is imaginary quadratic, then the
  representation $\rho_{A,\ell}:\Gal_\QQ\to\GL_4(\QQ_\ell)$ preserves
  two alternating pairings. The first of them is obtained from $K$ by
  taking the skew-symmetric form on a $2$-dimensional vector space,
  whose similitude character matches
  $\det\rho_\lambda = \eta_{L/K}\cdot \chi_\ell$. Since $\Gal(L/\QQ)$
  is abelian, the character $\eta_{L/K}$ extends to the full Galois
  group, and Proposition~\ref{prop:extension-induction} produces a
  non-degenerate alternating pairing on the induced representation
  $\rho_{A,\ell}=\Ind_{\Gal_\QQ}^{\Gal_K}\rho_\lambda$.
	
  On the other hand, the representation $\rho_{A,\ell}$ preserves the
  Weil pairing, which has similitude character $\chi_\ell$.
  The statement is consistent with the fact that many traces of $\rho_{A,\ell}$
  are zero (otherwise, the two similitude characters would coincide).

Here is an example of this phenomenon from
  \cite{MR2091964} (page 620). Consider the genus $2$ curve
  \[
\C: y^2 = (x^2+7)(83/30x^4+14x^3-1519/30x^2+49x-1813/120),
\]
and let $A$ denote its Jacobian. Then the following holds:
\begin{itemize}
\item $\End_L(A)$ is a maximal order in $B_6$ (the indefinite quaternion
  algebra of discriminant $6$), for $L=\QQ(\sqrt{-6},\sqrt{-14})$.
  
\item $\End_\QQ(A)=\ZZ$.
\item $\End_{\QQ(\sqrt{-6})}(A)\otimes_\ZZ \QQ =
  \QQ(\sqrt{-6})$. 
\end{itemize}

\end{remark}

\begin{corollary}
  Suppose that $A$ is principally polarizable and that $\End(A_L)$ is
  an hereditary order in $\End^0(A_L)$. Then $A$ is paramodular.
\end{corollary}
\begin{proof}
  By \cite[Theorem~3.4]{MR2091964}, our hypotheses imply that
  $\Gal(L/\QQ)\simeq C_2\times C_2$, so the Proposition~\ref{proposition:real-quadratic-field} implies that we
  can take $K$ such that $\End^0(A_K)$ is real quadratic in the proof
  of Theorem~\ref{theorem:potentialQM-implies-paramodular}, so the
  Nebentypus of $\rho_{A,\lambda}$ is trivial. Now the
  result follows from \cite{MR2887605} (main theorem) when $K$ is real
  quadratic and from \cite[Theorem 4.1]{MR3404028} when $K$ is
  imaginary quadratic.
\end{proof}

\bibliographystyle{alpha}
\bibliography{references}

\end{document}